\documentclass[11pt,a4paper]{article}
\usepackage{graphicx}
\usepackage[caption=false]{subfig}
\usepackage{latexsym,amsfonts,amsbsy,amssymb}
\usepackage{amsmath,amsthm,amsfonts}
\usepackage{cite}
\usepackage{color}
\usepackage{dsfont}
\usepackage{amssymb}
\usepackage[noblocks]{authblk}

\makeatletter

\@addtoreset{equation}{section} \makeatother
\newtheorem{theorem}{Theorem}[section]
\newtheorem{lemma}{Lemma}[section]

\newtheorem{remark}{Remark}[section]
\newtheorem{example}{Example}[section]

\makeatletter \@addtoreset{equation}{section} \makeatother
\allowdisplaybreaks % For long formula

\begin{document}
\title{\bf Anisotropic meshes and stabilized parameters for the stabilized finite element methods}
\date{}
\author{Yana Di\thanks{LSEC, NCMIS, Academy of Mathematics and Systems Science,
Chinese Academy of Sciences, Beijing 100190, China
({\tt yndi@lsec.cc.ac.cn}). The author is
supported in part by National Natural Science Foundation
of China (11271358)
and the National Center for Mathematics and Interdisciplinary Science,
CAS.
} \quad
Hehu Xie\thanks{LSEC, NCMIS, Academy of Mathematics and Systems Science,
Chinese Academy of Sciences, Beijing 100190, China
({\tt hhxie@lsec.cc.ac.cn}). The author is
supported in part by National Natural Science Foundation
of China (91330202, 11371026, 11001259,  2011CB309703)
and the National Center for Mathematics and Interdisciplinary Science,
CAS.
} \quad Xiaobo
Yin\thanks{Corresponding author, School of Mathematics and Statistics \& Hubei Key Laboratory of Mathematical Sciences, Central China Normal University, Wuhan 430079, China, ({\tt yinxb@mail.ccnu.edu.cn}).  The author is supported in part by the National
Natural Science Foundation of China (11201167).}}
\maketitle

\begin{abstract}
We propose a numerical strategy to generate the anisotropic meshes and select the appropriate stabilized parameters simultaneously for two dimensional convection-dominated convection-diffusion equations by stabilized continuous linear finite elements. Since the discretized error in a suitable norm can be bounded by the sum of interpolation error and its variants in different norms, we replace them by some terms which contain the Hessian matrix of the true solution, convective fields, and the geometric properties such as directed edges and the area of the triangle. Based on this observation, the shape, size and equidistribution requirements are used to derive the corresponding metric tensor and the stabilized parameters. It is easily found from our derivation that the optimal stabilized parameter is coupled with the optimal metric tensor on each element. Some numerical results are also provided to validate the stability and efficiency of the proposed numerical strategy.
\end{abstract}



\section{Introduction}
This paper is concerned with the numerical solution of the following scalar convection-diffusion equation
\begin{equation}\label{adv-diff-problem}
\left\{
\begin{array}{rcl}
-\varepsilon\Delta u+\mathbf b\cdot\nabla u&=&f,\ \ \ \  \ {\rm in}\ \Omega,\\
u&=&g,\ \ \ {\rm on}\ \partial \Omega,
\end{array}
\right.
\end{equation}
where $\Omega\subset \mathbb{R}^2$ is a bounded polygonal domain with boundary $\partial \Omega$, $\varepsilon>0$ is the constant diffusivity, ${\bf b}\in [W^{1,\infty}(\Omega)]^2$ is the given convective field satisfying the incompressibility condition
$\nabla \cdot {\bf b}=0$ in $\Omega$, $f\in L^2(\Omega)$ is the source function, and $g\in H^{1/2}(\partial \Omega)$ represents the
Dirichlet boundary condition.

Despite the apparent simplicity of problem (\ref{adv-diff-problem}), its
numerical solution become particularly challenging when convection dominates diffusion (i.e., when $\varepsilon\ll \|{\bf b}\|$). In such cases, the solution usually
exhibits very thin layers across which the derivatives of the solution
are large. The
widths of these layers are usually significantly smaller than
the mesh size and hence the layers can be hardly resolved.
 As a result of this, on
meshes which do not resolve the layers, standard Galerkin finite
element methods have poor stability and accuracy properties.

To enhance the stability and accuracy of the Galerkin
discretization in convection-dominated regime,
various stabilized strategies have been developed. Examples include
upwind scheme \cite{heinrich1977upwind}, streamline diffusion finite element methods (SDFEM), also known
as streamline-upwind Petrov-Galerkin formulation (SUPG) \cite{brooks1982streamline,hughes1987recent}, the Galerkin/least squares methods (GLS) \cite{hughes1989new}, residual-free-bubbles
(RFB) methods \cite{brezzi1998further,brezzi1999priori,cangiani2007residual,franca2007inf}, local projection schemes \cite{becker2001finite,braack2006local,knobloch2011stability,matthies2007unified,matthies2015local},   exponential fitting \cite{bank1990some,brezzi1989two,xu1999monotone}, discontinuous Galerkin methods \cite{baumann1999discontinuous,houston2002discontinuous}, subgrid-scale
techniques \cite{guermond1999stabilization,guermond2006subgrid},
continuous interior penalty methods \cite{braack2007stabilized,burman2005unified,burman2007continuous,dolejsi2013framework,el2007priori,ern2013weighting},
and spurious oscillations at layers diminishing (SOLD) methods (also known as shock capturing methods) \cite{john2007spurious,john2008spurious,knopp2002stabilized}, interested readers are referred to \cite{roos2008robust} for an extensive survey of the literature.

However, if uniform meshes are used for stabilized finite element methods, oscillations still exist near the layers in some cases although very fine meshes are used \cite{john2006computational}. Hence, it is more appropriate to generate adaptively anisotropic meshes to capture the layers. There are some recent efforts directed at constructing adaptive anisotropic
meshes which combine a stabilized scheme and some mesh
modification strategies. For example, the resolution of boundary
layers occurring in the singularly perturbed case is achieved using anisotropic
mesh refinement in boundary layer regions \cite{apel1996anisotropic}, where the actual choice of the element
diameters in the refinement zone and the determination of the numerical damping
parameters is addressed. In \cite{nguyen2009adaptive} an adaptive meshing algorithm
is designed by combining SUPG method, an adapted metric tensor and an anisotropic centroidal Voronoi tessellation algorithm, which is shown to be robust in detecting layers and efficient in avoiding non-physical oscillations
in the numerical approximation. Sun et al. \cite{sun2010numerical} develop a multilevel-homotopic-adaptive finite element method (MHAFEM)
by combining SDFEM, anisotropic mesh adaptation, and
the homotopy of the diffusion coefficient. The authors use numerical experiments to demonstrate
that MHAFEM can efficiently capture boundary or interior layers and produce accurate solutions.

This list is by no means exhausted, and there are many efforts to construct adaptive anisotropic
meshes by combining a stabilized scheme and some mesh
modification strategies. However, so far as we know, there are still two key problems which haven't been solved in rigorous ways.

First, although there are many results on optimal anisotropic meshes for minimizing the interpolation error and also the discretized error of finite element methods for solving the Laplace equation due to c\'{e}a's lemma, its extension to the discretized error of stabilized finite element methods for the convection-dominated convection-diffusion equation is not clear. So far as we know most results on optimal anisotropic meshes for the stabilized finite element methods applied to the convection-dominated convection-diffusion equation is just the same with that for the interpolation error. In fact, this strategy is not always optimal, which will be illustrated experimentally later in this paper.

Second, there is a crucial factor for most SFEMs, that is, the proper selection of the stabilization parameter $\alpha_K$ on the element $K$ (there are some attempts avoiding explicitly mesh-dependent stabilization parameter, e.g. \cite{bochev2015formulation,bochev2013parameter,franca1993element}). For example, the standard choice for quasi-uniform triangulations for SUPG is (\cite[P. 305-306]{roos2008robust})
\begin{equation*}
     \alpha_K=\left \{
     \begin{array}{lll}
     \alpha_0 h_K & Pe_K>1,& (\mbox{convection-dominated case})\\
     \alpha_1 h_K^2/\varepsilon & Pe_K\leq 1,& (\mbox{diffusion-dominated case})
     \end{array}
     \right .
\end{equation*}
with appropriate positive constants $\alpha_0$ and $\alpha_1$. Here $Pe_K:=\|{\bf b}\|_{0,\infty,K}h_K/(2\varepsilon)$ is the mesh Pecl\'{e}t
number and $h_K=\sup_{{\bf x},{\bf x}'\in K}\|{\bf x}-{\bf x}'\|$ is the diameter of the element $K$. A more sophisticated
choice is to replace the diameter $h_K$ of the element in the above definition for $\alpha_K$ by its
streamline diameter $h_{{\bf b},K}$ which is the maximal length of any characteristic running through $K$. How to extend the strategy for quasi-uniform triangulations to the case of anisotropic meshes? There are some attempts which basically use the analog of isotopic case to get the following form of stabilization parameters
\begin{equation}\label{stab_para_2d_old}
\alpha_K=\frac{h_K}{2\|{\bf b}\|_K}\min{\Big\{}1,\frac{Pe_K}{3}{\Big\}},\:
\mbox{with}\: Pe_K=\frac{\|{\bf b}\|_Kh_K}{2\varepsilon}.
\end{equation}
For example, Nguyen et al. \cite{nguyen2009adaptive} use stabilization parameters as the form of (\ref{stab_para_2d_old}) by setting $h_K$ as the length of the longest edge of the element $K$ projected onto the convective field $\bf b$, this choice of $h_K$ together with (\ref{stab_para_2d_old}) is denoted by ``LEP" in this paper. Another similar choice of $h_K$ is the length of the projection of the longest edge of the element $K$ onto the convective field $\bf b$, which together with (\ref{stab_para_2d_old}) is denoted by ``PLE". As pointed in isotropic case, a more sophisticated
choice is to choose $h_K$ as the maximal length of any characteristic running through $K$, i.e., the diameter of $K$ in the direction of the convection ${\bf b}$, which together with (\ref{stab_para_2d_old}) is denoted by ``DDC".
There are also some other choices of stabilized parameters derived by relatively rigorous theory on anisotropic meshes. For example, in \cite{apel1996anisotropic} an anisotropic a priori error analysis is provided for the advection-diffusion-reaction problem. It is
shown that the diameter of each element $K$ (together with (\ref{stab_para_2d_old}), it is denoted by ``DEE"), should
be used as $h_K$ in (\ref{stab_para_2d_old}) for the design of the stability parameters in the case of external boundary layers. In \cite{kunert2001robust} an alternative approach is proposed showing
that the diameter of each element is again the correct choice. Micheletti et al. \cite{micheletti2003stabilized} consider the GLS methods for the scalar advection-diffusion and the Stokes problems with
approximations based on continuous piecewise linear finite elements on anisotropic
meshes, where new definitions of the stability parameters are proposed.
Cangiani and S\"{u}li \cite{cangiani2007residual} use the stabilizing term derived from the
RFB method to redefine the mesh P\'{e}clet number and propose a new choice of the
streamline-diffusion parameters (This well-known fact that the RFB method and the SDFEM are equivalent under certain conditions was first observed by Brezzi and Russo \cite{brezzi1994choosing}. The similar idea was also used in \cite{linss2005anisotropic}) that is suitable for use on anisotropic partitions. Although there are so many strategies on selection of the stabilization parameters, it is still hard to show which is optimal. Besides, there is little result on the relationship between the strategy to generate the anisotropic meshes and the selection of stabilized parameters.

In this paper, we propose a strategy to generate the anisotropic meshes and select the appropriate stabilized parameters simultaneously of stabilized continuous linear finite elements for two dimensional convection-dominated problems. As in \cite{franca1992stabilized,hughes1989new,micheletti2003stabilized}, the discretized error (the difference between the true solution and the stabilized finite element solution) in a suitable norm can be bounded by the sum of interpolation error and its variants in different norms. Based on this result, we use the idea in our recent work \cite{xie2014metric} to replace these norms of interpolation error by some terms which contain the Hessian matrix of the true solution, convective fields ${\bf b}$, and the geometric properties such as directed edges and the area of the triangle. After that, we use the shape, size and equidistribution
requirements to derive the correspond metric tensor and the stabilized parameters. From our derivation it is easily found that the optimal stabilized parameter is coupled with the optimal metric tensor on each element. Specifically, the relationship between the optimal metric tensor and the optimal stabilized parameter on each element is given approximately by (\ref{relation_C_K_Alpha}).

The rest of the paper is organized as follows. In Section \ref{Section:Stab_FEM}, we state the GLS stabilized finite element method for the problem (\ref{adv-diff-problem}).
Section \ref{Section:Basic_Estimate} is devoted to obtaining the estimate for the discretized error in a suitable norm via the anisotropic framework similar to that used in \cite{xie2014metric}. The optimal choice of the metric tensor and the stabilized parameters for the stabilized linear finite element methods are then derived in Section \ref{Section_Metric_Tensor}. Some numerical examples are provided in Section \ref{Section_Numerical_Examples} to demonstrate the stability and efficiency of the proposed numerical strategy. Some concluding remarks will be given in the last section.

\section{Stabilized finite element discretization}\label{Section:Stab_FEM}
We shall use the standard notations in  for the
Sobolev spaces $H^s(\Omega)$ and their associated inner products
$(\cdot,\cdot)_s$, norms $||\cdot||_s$, and seminorms $|\cdot|_s$
for $s \geq 0$. The Sobolev space $H^0(\Omega)$ coincides with
$L^2(\Omega)$, in which case the norm and inner product are denoted
by $||\cdot||$ and $(\cdot,\cdot)$, respectively. Let $H^1_g(\Omega) = \{v \in H^1(\Omega), v|_{\partial\Omega}=g\}$ and $H^1_0(\Omega) = \{v \in H^1(\Omega), v|_{\partial\Omega}=0\}$. The variational formulation of (\ref{adv-diff-problem}) reads as follows: find $u\in H^1_g(\Omega)$
which satisfies
\begin{equation}\label{adv-diff-weak-problem}
A(u,v)=F(v),\quad \forall v\in H^1_0(\Omega),
\end{equation}
where $A(\cdot,\cdot)$ and $F(\cdot)$ define the bilinear and linear forms
\begin{equation*}
A(u,v)=(\varepsilon \nabla u,\nabla v)+({\bf b}\cdot \nabla u,v),
\end{equation*}
and
\begin{equation*}
F(v)=(f,v),
\end{equation*}
respectively.

Given a triangulation $\mathcal{T}_h$ of $\Omega$, we denote the piecewise linear and continuous finite element space by $V^h$, i.e.,
\begin{equation*}
V^h = \{v \in H^1(\Omega), v|_K\in \mathcal{P}_1(K), \forall K\in \mathcal{T}_h\},
\end{equation*}
where $\mathcal{P}_1(K)$ is linear polynomial space in one element $K$. We then define $V^h_g:= V^h\bigcap H^1_g(\Omega)$ and $V^h_0:= V^h\bigcap H^1_0(\Omega)$. The standard finite element discretization of (\ref{adv-diff-weak-problem}) is to find $u_h\in V^h_g$ such that
\begin{equation}\label{femdiscretization}
A(u_h, v_h) = f(v_h),\quad \forall v_h\in V^h_0.
\end{equation}
For convection-dominated problems ($\varepsilon\ll \|{\bf  b}\|$)
, (\ref{femdiscretization}) using standard grid sizes are not
able to capture steep layers without introducing non-physical oscillations.
To enhance the stability and accuracy in the convection dominated regime,
various stabilization strategies have been developed. Here we take the GLS stabilized finite element method as an example which reads
as follows: find $u_h\in V^h_g$ such that
\begin{equation}\label{adv-diff-stab}
A_h(u_h,v_h)=F(v_h),\quad \forall v_h\in V^h_0,
\end{equation}
with
\begin{equation*}
A_h(u_h,v_h)=A(u_h,v_h)+\sum_{K\in \mathcal{T}_h}\alpha_K(-\varepsilon \Delta u_h+{\bf b}\cdot \nabla u_h,-\varepsilon \Delta v_h+{\bf b}\cdot \nabla v_h)_K,
\end{equation*}
and
\begin{equation*}
F_h(v_h)=F(v_h)+\sum_{K\in \mathcal{T}_h}\alpha_K(f,-\varepsilon \Delta v_h+{\bf b}\cdot \nabla v_h)_K.
\end{equation*}
In this paper we use linear finite element method, so the terms $\Delta u_h|_K$ and $\Delta v_h|_K$ in the two above equations are identically equal to zero. At this time, the GLS approach is the same as the SUPG method, which also enjoys the result in this paper if the linear finite element method is used. We endow the space $H^1_0(\Omega)$ with
the discrete norm $\|\cdot\|_h$ defined, for any $w\in H^1_0(\Omega)$, by
\begin{equation}\label{Energy_Norm}
\|w\|_h^2:=\varepsilon\|\nabla w\|_{L^2(\Omega)}^2
+\sum_{K\in\mathcal{T}_h}\alpha_K\|\mathbf b\cdot\nabla
w\|_{L^2(K)}^2.
\end{equation}

\begin{lemma}\label{Lemma:2derrorbound}
The stabilized finite element approximation $u_h$ defined by (\ref{adv-diff-stab})
has the following error estimate
\begin{align}\label{2derrorbound}
\|u-u_h\|_h^2&\leq C\sum_{K\in\mathcal{T}_h}{\Big (}\alpha_K^{-1}\|u-\Pi_hu\|_{L^2(K)}^2 + \varepsilon\|\nabla (u-\Pi_hu)\|_{L^2(K)}^2\nonumber\\
&+\alpha_K\|{\bf b}\cdot\nabla (u-\Pi_hu)\|_{L^2(K)}^2 + \alpha_K\varepsilon^2\|\Delta(u-\Pi_hu)\|_{L^2(K)}^2{\Big )},
\end{align}
where $\Pi_h$ denotes the standard continuous piecewise linear interpolation operator.
\end{lemma}
\begin{proof}
See, for example, \cite{franca1992stabilized,hughes1989new,micheletti2003stabilized}.
\end{proof}

\section{Estimates for the interpolation error and its
variants}\label{Section:Basic_Estimate}
As stated in Lemma \ref{Lemma:2derrorbound}, the discretized error in $\|\cdot\|_h$ norm is bounded by four terms of interpolation error and its variants in different norms. In fact the
interpolation error depends on the solution, the size and shape of
the elements in the mesh. Understanding this relation is crucial for
generating efficient meshes. In the mesh generation fields, this
relation is often studied for the model problem of interpolating
quadratic functions. For instance, Nadler \cite{nadler1986piecewise} derived an exact
expression for the $L^2$-norm of the linear interpolation error in
terms of the three sides ${\bf \ell}_1$, ${\bf \ell}_2$, and ${\bf
\ell}_3$ of the triangle $K$:
\begin{equation}\label{Nadlerformula}
\|u-\Pi_hu\|^2_{L^2(K)}=\frac{|K|}{180}{\Big[}{\Big(}d_{11}+d_{22}+d_{33}{\Big)}^2+d_{11}^2+d_{22}^2+d_{33}^2{\Big]},
\end{equation}
where $|K|$ is the area of the triangle, $d_{ij}={\bf \ell}_i\cdot
H(u){\bf \ell}_j$ with $H(u)$ being the Hessian of $u$.

Three element-wise error estimates in different norms are derived by
the following lemmas, which, together with \cite{nadler1986piecewise}, are fundamental for further discussion. Suppose $u$ is a quadratic function on a triangle $K$. The function is given
by its matrix representation:
\begin{equation*}
\forall {\bf x}\in K,\quad u({\bf x})=\frac{1}{2}{\bf x}^tH(u){\bf x}.
\end{equation*}

\begin{lemma}\label{Lemma:2D_H1}
Let $u$ be a quadratic function on a triangle $K$, and $\Pi_hu$ be
the Lagrangian linear interpolation of $u$ on $K$. The following
relationship holds:
\begin{equation}\label{2destimator_H1}
\|\nabla(u-\Pi_hu)\|^2_{L^2(K)}=\frac{1}{48|K|}\sum_{\substack{i,j=1,2\\i\leq
j}}D_{ij}{\bf \ell}_i^t{\bf
\ell}_j,
\end{equation}
where
\begin{equation}\label{D_ij}
D_{11}=d_{12}^2+d_{23}^2, D_{22}=d_{12}^2+d_{13}^2,
D_{12}=2d_{12}^2.
\end{equation}
\end{lemma}
\begin{proof}
The proof is similar to but easier than that of Lemma \ref{Lemma:2D_Conv}.
\end{proof}

\begin{remark}
Lemma \ref{Lemma:2D_H1} is also used in \cite{xie2014metric}.
\end{remark}

\begin{lemma}\label{Lemma:2D_Conv}
Let $u$ be a quadratic function on a triangle $K$, and $\Pi_hu$ be the Lagrangian linear
interpolation of $u$ on $K$. Assume ${\bf b}$ is a constant vector on $K$, the following relationship
holds:
\begin{equation}\label{2destimator_Conv}
\|\mathbf
b\cdot\nabla(u-\Pi_hu)\|^2_{L^2(K)}=\frac{1}{48|K|}\sum_{\substack{i,j=1,2\\i\leq
j}}D_{ij}k_ik_j,
\end{equation}
where
\begin{equation}\label{k1k2}
k_{1}=(b_{2},-b_{1}){\bf
\ell}_{1},\quad
k_{2}=(b_{2},-b_{1}){\bf
\ell}_{2}.
\end{equation}
\end{lemma}
\begin{proof}
Following \cite{loseille2011continuous}, we first define the reference element $K_r$
by its three vertices coordinates:
\begin{equation*}
\hat{\bf x}_1=(0,0)^t, \hat{\bf x}_2=(1,0)^t,\,\, \mbox{and}\,\,
\hat{\bf x}_3=(0,1)^t.
\end{equation*}
All the terms are computed on $K_r$ and then converted onto the
element $K$ at hand by using the following affine mapping:
\begin{equation*}
{\bf x}={\bf x}_1+B_K\hat{\bf x}\,\, \mbox{with}\,\, B_K=[{\bf
\ell}_1, -{\bf \ell}_2],\,\, {\bf x}\in K, \,\,\hat{\bf x}\in K_r,
\end{equation*}
where
\begin{equation*}
{\bf \ell}_1={\bf x}_2-{\bf x}_1,\,\,\mbox{and}\,\,  \, {\bf
\ell}_2={\bf x}_1-{\bf x}_3.
\end{equation*}
In the frame of $K_r$, the quadratic function $u$ turns into:
\begin{equation*}
u({\bf x}(\hat{\bf x}))=\frac{1}{2}{\bf x}_1^tH(u){\bf
x}_1+\frac{1}{2}{\bf x}_1^tH(u)B_K\hat{\bf x} +\frac{1}{2}\hat{\bf
x}^tB_K^tH(u){\bf x}_1+\frac{1}{2}\hat{\bf x}^tB_K^tH(u)B_K\hat{\bf x}.
\end{equation*}
Since the linear interpolation is concerned, linear and constant
terms of $u({\bf x}(\hat{\bf x}))$ are exactly interpolated,
these terms are neglected and only quadratic terms are kept. So we
could set $u({\bf x})=\frac{1}{2}\hat{\bf x}^tB_K^tH(u)B_K\hat{\bf
x}$, with a matrix form:
\begin{equation*}
u({\bf x}(\hat{\bf x}))=\frac{1}{2}\hat{\bf x}^tB_K^tH(u)B_K\hat{\bf
x}=\frac{1}{2} \left(\begin{array}{c}
 \hat{x} \\
 \hat{y}
\end{array}\right)^t
\left[\begin{array}{cc}
d_{11} & -d_{12} \\
-d_{21} & d_{22}
\end{array}\right]
\left(\begin{array}{c}
 \hat{x} \\
 \hat{y}
\end{array}\right).
\end{equation*}
Then the exact point-wise interpolation error of the function $u$ in $K_r$ reads:
\begin{equation}\label{2derrorexpan}
e({\bf x}(\hat{\bf x}))=\frac{1}{2}
[d_{11}(\hat{x}^2-\hat{x})+d_{22}(\hat{y}^2-\hat{y})-2d_{12}\hat{x}\hat{y}].
\end{equation}
It is obvious that the following formulas hold
\begin{equation*}
B_K=[{\bf \ell}_1, -{\bf \ell}_2]= \left[\begin{array}{cc}
x_{2}-x_{1} & x_{3}-x_{1} \\
y_{2}-y_{1} & y_{3}-y_{1}
\end{array}\right], \quad
B_K^{-1}=\frac{1}{\det(B_K)} \left[\begin{array}{cc}
y_{3}-y_{1} & x_{1}-x_{3} \\
y_{1}-y_{2} & x_{2}-x_{1}
\end{array}\right].
\end{equation*}
After that it is easily to obtain
\begin{equation*}
\det(B_K)\cdot{\mathbf b}^{t}B_k^{-t}={\mathbf b}^{t}
\left[\begin{array}{cc}
y_{3}-y_{1} & y_{1}-y_{2} \\
x_{1}-x_{3} & x_{2}-x_{1}
\end{array}\right]=(b_{2},-b_{1})[{\bf
\ell}_{2}, {\bf \ell}_{1}]=(k_{2},k_{1}).
\end{equation*}
Then we have
\begin{align}\label{2d_derivation}
\int_K{\big(}\mathbf b\cdot\nabla_{\bf x}e({\bf x}){\big)}^2dxdy &=
\det(B_K)\int_{K_r}{\Big(}{\mathbf b}^{t} B_K^{-t}\nabla_{\hat{\bf x}}e{\big(}{\bf x}(\hat{\bf x}){\big)}{\Big)}^2d\hat{x}d\hat{y}\nonumber\\
&=\frac{1}{\det(B_K)}{\Bigg [}k_{2}^2\int_{K_r}{\Big (}\frac{\partial e({\bf x}(\hat{\bf x}))}{\partial \hat{x}}{\Big )}^2d\hat{x}d\hat{y}
+k_{1}^2\int_{K_r}{\Big (}\frac{\partial e({\bf x}(\hat{\bf x}))}{\partial \hat{y}}{\Big )}^2d\hat{x}d\hat{y}\nonumber\\
&+2k_{1}k_{2}\int_{K_r}\frac{\partial e({\bf x}(\hat{\bf
x}))}{\partial \hat{x}}\frac{\partial e({\bf x}(\hat{\bf
x}))}{\partial \hat{y}}d\hat{x}d\hat{y}{\Bigg ]}.
\end{align}
Due to (\ref{2derrorexpan}), we can easily obtain
\begin{equation*}
\nabla_{\hat{\bf x}}e({\bf x}(\hat{\bf x}))=\left(\begin{array}{c}
\partial e({\bf x}(\hat{\bf x}))/\partial \hat{x} \\
\partial e({\bf x}(\hat{\bf x}))/\partial \hat{y}
\end{array}\right)=
\frac{1}{2} \left(\begin{array}{c}
d_{11}(2\hat{x}-1)-2d_{12}\hat{y} \\
d_{22}(2\hat{y}-1)-2d_{12}\hat{x}
\end{array}\right).
\end{equation*}
After simple calculation, the following results hold:
\begin{align*}
  && \int_{K_r} \hat{x}^2d\hat{x}d\hat{y}=\int_{K_r} \hat{y}^2d\hat{x}d\hat{y}=\frac{1}{12},\ \ \
   \int_{K_r}\hat{x} \hat{y}d\hat{x}d\hat{y}=\frac{1}{24},\nonumber\\
  &&\int_{K_r} \hat{x}d\hat{x}d\hat{y}=\int_{K_r} \hat{y}d\hat{x}d\hat{y}=\frac{1}{6},\ \ \
  \int_{K_r}1d\hat{x}d\hat{y}=\frac{1}{2}.
\end{align*}
Then we have
\begin{equation}\label{2dpart1}
24\int_{K_r}{\Big (}\frac{\partial e({\bf x}(\hat{\bf x}))}{\partial
\hat{x}}{\Big )}^2d\hat{x}d\hat{y}= d_{12}^2+d_{13}^2=D_{22},
\end{equation}

\begin{equation}\label{2dpart2}
24\int_{K_r}{\Big (}\frac{\partial e({\bf x}(\hat{\bf x}))}{\partial
\hat{y}}{\Big )}^2d\hat{x}d\hat{y}= d_{12}^2+d_{23}^2=D_{11},
\end{equation}

\begin{equation}\label{2dpart3}
48\int_{K_r}\frac{\partial e({\bf x}(\hat{\bf x}))}{\partial
\hat{x}}\frac{\partial e({\bf x}(\hat{\bf x}))}{\partial
\hat{y}}d\hat{x}d\hat{y}=2d_{12}^2=D_{12}.
\end{equation}
Substituting (\ref{2dpart1}), (\ref{2dpart2}), (\ref{2dpart3}) into
(\ref{2d_derivation}) we get the desired results
(\ref{2destimator_Conv}) due to the fact $\det(B_k)=2|K|$.
\end{proof}

\begin{lemma}\label{Lemma:2D_Lapl}
Let $u$ be a quadratic function on a triangle $K$. The following
relationship holds:
\begin{equation}\label{2destimator_Lapl}
\|\Delta u\|^2_{L^2(K)}=\frac{{\Big(}{|\bf \ell}_{2}|^2d_{11}-2{\bf
\ell}_{1}^t{\bf \ell}_{2}d_{12}+{|\bf
\ell}_{1}|^2d_{22}{\Big)}^2}{16|K|^3}.
\end{equation}
\end{lemma}
\begin{proof}
Similar to that of Lemma \ref{Lemma:2D_Conv} we can easily obtain
\begin{align*}
\int_K(\Delta u)^2dxdy &=
\det(B_K)\int_{K_r}{\Bigg(}\frac{{|\bf \ell}_{2}|^2}{\det(B_K)^2}d_{11}-2\frac{{\bf \ell}_{1}^t{\bf \ell}_{2}}{\det(B_K)^2}d_{12}+\frac{{|\bf \ell}_{1}|^2}{\det(B_K)^2}d_{22}{\Bigg)}^2d\hat{x}d\hat{y}\\
&=\frac{{\Big(}{|\bf \ell}_{2}|^2d_{11}-2{\bf \ell}_{1}^t{\bf
\ell}_{2}d_{12}+{|\bf
\ell}_{1}|^2d_{22}{\Big)}^2}{2\det(B_K)^3}=\frac{{\Big(}{|\bf \ell}_{2}|^2d_{11}-2{\bf \ell}_{1}^t{\bf
\ell}_{2}d_{12}+{|\bf \ell}_{1}|^2d_{22}{\Big)}^2}{16|K|^3}.
\end{align*}
This is the desired result (\ref{2destimator_Lapl}) and the proof is complete.
\end{proof}
\begin{remark}
Since $\Delta u$ is a constant under our assumption, there is a rather direct and easy way to prove Lemma \ref{Lemma:2D_Lapl}. However, we still use the frame of proof for Lemma \ref{Lemma:2D_Conv} to make the error expression be a consistent manner.
\end{remark}

\begin{theorem}\label{Thm2derrorbound}
Assume the exact solution $u$ is quadratic on each element $K$, the error of the
stabilized finite element approximation has the following estimate
\begin{equation}\label{Error_Estimate_In_Energy_Norm}
\|u-u_h\|_h^2\leq  C\sum_{K\in\mathcal{T}_h}E_K,
\end{equation}
with
\begin{align*}
E_K&=\frac{|K|}{180\alpha_K}{\Bigg[}{\Big(}\sum_{i=1}^{3}d_{ii}{\Big)}^2+\sum_{i=1}^{3}d_{ii}^2{\Bigg]}
+ \frac{\varepsilon}{48|K|}\sum_{\substack{i,j=1,2\\i\leq
j}}D_{ij}{\bf \ell}_i^t{\bf
\ell}_j
 + \frac{\alpha_K}{48|K|}\sum_{\substack{i,j=1,2\\i\leq
j}}D_{ij}k_ik_j\\
&+\frac{\alpha_K\varepsilon^2}{16|K|^3}{\Big(}{|\bf
\ell}_{2}|^2d_{11}-2{\bf \ell}_{1}^t{\bf \ell}_{2}d_{12}+{|\bf
\ell}_{1}|^2d_{22}{\Big)}^2,
\end{align*}
where $D_{ij}$ and $k_i$ are defined by (\ref{D_ij}) and (\ref{k1k2}), respectively.
\end{theorem}

\begin{proof}
Together with Lemmas \ref{Lemma:2derrorbound}, \ref{Lemma:2D_H1}, \ref{Lemma:2D_Conv} and \ref{Lemma:2D_Lapl}, the conclusion is obtained directly.
\end{proof}
\noindent
Even the error estimate (\ref{Error_Estimate_In_Energy_Norm}) is only valid for those piecewise quadratic functions, however, it could catch the
main properties of the errors for general functions.
In fact, the treatment to replace the general solution by its second order Taylor expansion yields a reliable and efficient
estimator of the interpolation error for general functions provided
a saturation assumption is valid
\cite{agouzal2010minimization,dorfler2002small}.

For simplicity of notation in the following discussion, $\forall K\in\mathcal{T}_h$, we denote by
\begin{equation*}
Q_{1,K}=\frac{|K|}{180}{\Bigg[}{\Big(}\sum_{i=1}^{3}d_{ii}{\Big)}^2+\sum_{i=1}^{3}d_{ii}^2{\Bigg]},\quad
Q_{2,K}=\frac{1}{48|K|}\sum_{\substack{i,j=1,2\\i\leq
j}}D_{ij}{\bf \ell}_i^t{\bf
\ell}_j,
\end{equation*}
\begin{equation*}
\widetilde{Q}_{2,K}=\frac{1}{48|K|}\sum_{\substack{i,j=1,2\\i\leq
j}}D_{ij}k_ik_j,\quad
Q_{3,K}=\frac{1}{16|K|^3}{\Big(}{|\bf
\ell}_{2}|^2d_{11}-2{\bf \ell}_{1}^t{\bf \ell}_{2}d_{12}+{|\bf
\ell}_{1}|^2d_{22}{\Big)}^2.
\end{equation*}
And then $E_K$ in (\ref{Error_Estimate_In_Energy_Norm}) can be recast into
\begin{equation*}
E_K=Q_{1,K}\cdot\alpha_K^{-1} + Q_{2,K}\cdot\varepsilon +
\widetilde{Q}_{2,K}\cdot\alpha_K
+Q_{3,K}\cdot \alpha_K\varepsilon^2.
\end{equation*}

\section{Metric tensors for anisotropic mesh adaptation}\label{Section_Metric_Tensor}
We now use the error estimates obtained in Section \ref{Section:Basic_Estimate} to develop the metric tensor for the discretized error in $\|\cdot\|_h$ norm and give a new definition of the stability parameters which are optimal in a certain sense. In the field of mesh generation, the metric tensor, $\mathcal{M}({\bf x})$, is commonly used such that an anisotropic mesh is generated as a quasi-uniform mesh in the metric space determined by
$\mathcal{M}({\bf x})$. Mathematically, this can be interpreted as
the following three requirements (\cite{huang2005metric}). {\bf 1. The shape requirement.} The elements of the new mesh, $\mathcal{T}_h$, are (or are close to being) equilateral in the metric. {\bf 2. The size requirement.} The elements of the new mesh
$\mathcal{T}_h$ have a unitary volume in the metric. {\bf 3. The equidistribution requirement. } The anisotropic mesh is
required to minimize the error for a given number of mesh points (or
equidistribute the error on every element).

\subsection{Optimal metric tensor and stabilized parameters}

\begin{figure}[ht]
  \centering
  \includegraphics[width=12cm]{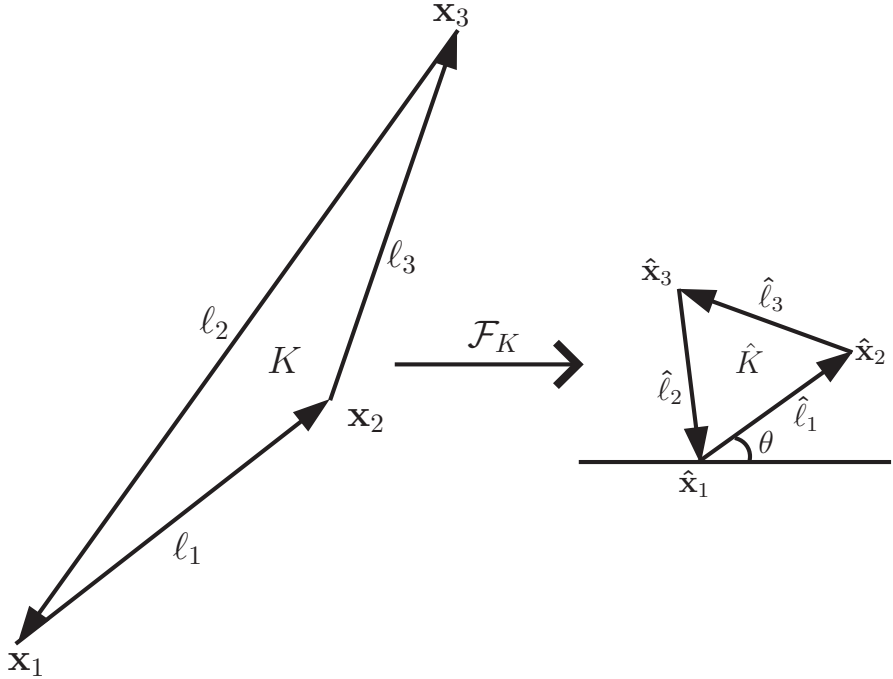}
  \put(-165,120){\Large{$\mathcal{F}_K$}}
  \put(-337,-2){\Large{${\bf x}_1$}} \put(-210,90){\Large{${\bf x}_2$}} \put(-178,243){\Large{${\bf x}_3$}}
  \put(-240,110){\Large{$K$}} \put(-280,40){\Large{ ${\bf \ell}_{1}$}} \put(-270,125){\Large{ ${\bf \ell}_{2}$}}
  \put(-200,150){\Large{ ${\bf \ell}_{3}$}}
  \put(-86,65){\large{${\bf \hat{x}}_1$}} \put(-20,115){\large{${\bf \hat{x}}_2$}} \put(-100,145){\large{${\bf \hat{x}}_3$}}
  \put(-65,110){\large{$\hat{K}$}} \put(-47,90){\large{ ${\bf \hat{\ell}}_{1}$}}
  \put(-98,100){\large{ ${\bf \hat{\ell}}_{2}$}} \put(-60,135){\large{ ${\bf \hat{\ell}}_{3}$}}
  \put(-60,80){\large{ ${\bf \theta}$}}
  \caption{Affine map ${\bf \hat{x}}=\mathcal{F}_K{\bf x}$ from triangle $K$ to the
reference triangle $\hat{K}$.}\label{Affine_map_2d}
  \end{figure}
We derive the monitor function $M({\bf x})$ first. At this time, we just need the shape
and equidistribution requirements. Assume $H(u)$ is
a symmetric positive definite matrix on every point ${\bf x}$ and
this restriction can be explained by Remark 2 in \cite{loseille2011continuous}. Set
$M({\bf x})=C({\bf x})H(u)$. Denoted by $H_K$ and $C_K$ the $L^2$
projection of $H(u)$ and $C({\bf x})$ to the constant space on $K$, and
$M_{K}=C_KH_K$. Since $H_K$ is a symmetric positive definite matrix,
we do the singular value decomposition $H_K=R^T\Lambda R$, where
$\Lambda=\mbox{diag}(\lambda_{1,K},\lambda_{2,K})$ is the diagonal matrix of
the corresponding eigenvalues ($\lambda_{1,K},\lambda_{2,K}>0$) and $R$ is
the orthogonal matrix with rows being the eigenvectors of $H_K$.
Denote by $F_K$ and ${\bf t}_K$ the matrix and the vector defining
the invertible affine map $\hat{\bf x}=\mathcal{F}_K({\bf
x})=F_K{\bf x} + {\bf t}_K$ from the generic element $K$ to the
reference triangle $\hat{K}$ (see Figure \ref{Affine_map_2d}). Here
we take $\hat{K}$ as an equilateral triangle with one edge which has angle $\theta$ with the horizontal line. Let $M_{K}=F_K^TF_K$, then
$F_K=C_K^{\frac{1}{2}}\Lambda^{\frac{1}{2}} R$. Mathematically, the shape
requirement can be expressed as
\begin{equation}\label{shape2d}
|\hat{\ell}_1|=|\hat{\ell}_2|=|\hat{\ell}_3|=L,\quad
\frac{\hat{\ell}_{1}\cdot\hat{\ell}_{3}}{|\hat{\ell}_{1}|\cdot|\hat{\ell}_{3}|}=
\frac{\hat{\ell}_{2}\cdot\hat{\ell}_{3}}{|\hat{\ell}_{2}|\cdot|\hat{\ell}_{3}|}=
\frac{\hat{\ell}_{1}\cdot\hat{\ell}_{2}}{|\hat{\ell}_{1}|\cdot|\hat{\ell}_{2}|}=\cos(2\pi/3),
\end{equation}
where $L$ is a constant.
\begin{theorem}\label{theorem:Q_iK}
Under the shape
requirement, the following results hold:
\begin{equation*}
Q_{1,K}=\frac{L^4|K|}{15C_K^2},\quad
Q_{2,K}=\frac{L^4{\rm tr}(H_K)}{32\sqrt{3}C_K^2\det(H_K)^{\frac{1}{2}}},\quad
Q_{3,K}=\frac{3L^4{\rm tr}(H_K)^2}{16C_K^2|K|\det(H_K)},
\end{equation*}
and
\begin{equation*}
\widetilde{Q}_{2,K}=\frac{L^4\det(H_K)^{\frac{1}{2}}}{24C_K^2}
{\Big(}\frac{A_1^2}{\lambda_{1,K}}+\frac{A_2^2}{\lambda_{2,K}}{\Big)}=
\frac{L^4}{32\sqrt{3}C_K^2}\cdot\frac{{\bf b}^tH_K{\bf b}}{\det(H_K)^{\frac{1}{2}}},
\end{equation*}
where
\begin{equation*}
 {\bf A}=\left(\begin{array}{c}
 A_1 \\
 A_2
\end{array}\right)=
R\left[\begin{array}{cc}
0 & 1 \\
-1 & 0
\end{array}\right]
{\bf b}.
\end{equation*}
\end{theorem}
\begin{proof}
From the definition for $Q_{1,K}$ and the relation between $\ell_i$ and $\hat{\ell}_i$, we have
\begin{equation*}
Q_{1,K}=\frac{|K|}{180C_K^2}{\Big[}{\Big(}\sum_{i=1}^{3}|\hat{\ell}_i|^2{\Big)}^2
+\sum_{i=1}^{3}|\hat{\ell}_{i}|^4{\Big]}
=\frac{L^4|K|}{15C_K^2}.
\end{equation*}
Similarly, using (\ref{shape2d}), the following reduction is straight-forward:
\begin{align}\label{Q2total}
&Q_{2,K}=\frac{1}{48|K|}\sum_{\substack{i,j=1,2\\i\leq
j}}D_{ij}{\bf \ell}_i^t{\bf
\ell}_j=
\frac{1}{48|K|}{\Big(}|{\bf \ell}_{1}|^2(d_{12}^2+d_{23}^2)+|{\bf \ell}_{2}|^2(
d_{12}^2+d_{13}^2)+2({\bf \ell}_{1}\cdot {\bf \ell}_{2})d_{12}^2{\Big)}\nonumber\\
&=\frac{1}{48C_K^2|K|}{\Big(}|{\bf \ell}_{1}|^2{\big(}({\bf \hat{\ell}}_{1}\cdot {\bf
\hat{\ell}}_{2})^2+({\bf \hat{\ell}}_{2}\cdot {\bf \hat{\ell}}_{3})^2{\big
)}+|{\bf \ell}_{2}|^2{\big(}({\bf \hat{\ell}}_{1}\cdot {\bf \hat{\ell}}_{2})^2+({\bf
\hat{\ell}}_{1}\cdot {\bf
\hat{\ell}}_{3})^2{\big )}+2({\bf \ell}_{1}\cdot {\bf \ell}_{2})({\bf \hat{\ell}}_{1}\cdot {\bf \hat{\ell}}_{2})^2{\Big)}\nonumber\\
&=\frac{L^4}{48C_K^2|K|}{\Big(}|{\bf \ell}_{1}|^2+|{\bf \ell}_{2}|^2+({\bf \ell}_{1}\cdot {\bf \ell}_{2}){\Big)}\cdot 2\cos^2\Big(\frac{2\pi}{3}\Big)=\frac{L^4}{96C_K^2|K|}\sum_{\substack{i,j=1,2\\i\leq
j}}{\bf \ell}_i^t{\bf
\ell}_j\nonumber\\
&=\frac{L^4\det(H_K)^{\frac{1}{2}}}{96C_K|\hat{K}|}\sum_{\substack{i,j=1,2\\i\leq
j}}{\Big(}C_K^{-\frac{1}{2}}R^{-1}\Lambda^{-\frac{1}{2}}\hat{\ell}_i{\Big)}\cdot
{\Big(}C_K^{-\frac{1}{2}}R^{-1}\Lambda^{-\frac{1}{2}}\hat{\ell}_j{\Big)}\quad\quad\text{by $|K|=\frac{|\hat{K}|}{C_K\sqrt{\det(H_K)}}$}\nonumber\\
&=\frac{L^4\det(H_K)^{\frac{1}{2}}}{96C_K^2|\hat{K}|}\sum_{\substack{i,j=1,2\\i\leq
j}}{\Big(}\Lambda^{-\frac{1}{2}}\hat{\ell}_i{\Big)}\cdot{\Big(}\Lambda^{-\frac{1}{2}}\hat{\ell}_j{\Big)},
\end{align}
Since $\hat{\ell}_1=L(\cos\theta,\sin\theta)^t$ and $\hat{\ell}_2=-L{\Big(}\cos{\big(}\frac{\pi}{3}+\theta{\big )},\sin{\big(}\frac{\pi}{3}+\theta{\big)}{\Big)}^t$, after simple calculation,
\begin{equation*}
\Lambda^{-\frac{1}{2}}\hat{\ell}_1=L{\Big(}\lambda_{1,K}^{-\frac{1}{2}}\cos\theta,\lambda_{2,K}^{-\frac{1}{2}}\sin\theta{\Big)}^t,
\: \Lambda^{-\frac{1}{2}}\hat{\ell}_2=-L{\Big(}\lambda_{1,K}^{-\frac{1}{2}}\cos\big(\frac{\pi}{3}+\theta\big),
\lambda_{2,K}^{-\frac{1}{2}}\sin\big(\frac{\pi}{3}+\theta\big){\Big)}^t,
\end{equation*}
and then the following three equalities hold:
\begin{align*}
{\Big(}\Lambda^{-\frac{1}{2}}\hat{\ell}_1{\Big)}\cdot{\Big(}\Lambda^{-\frac{1}{2}}\hat{\ell}_1{\Big)}
&=L^2{\Big(}\lambda_{1,K}^{-1}\cos^2\theta+\lambda_{2,K}^{-1}\sin^2\theta{\Big)},\\
{\Big(}\Lambda^{-\frac{1}{2}}\hat{\ell}_2{\Big)}\cdot{\Big(}\Lambda^{-\frac{1}{2}}\hat{\ell}_2{\Big)}
&=L^2{\Big(}\lambda_{1,K}^{-1}\cos^2{\big(}\frac{\pi}{3}+\theta{\big)}+\lambda_{2,K}^{-1}\sin^2\big(\frac{\pi}{3}+\theta\big){\Big)},\\
{\Big(}\Lambda^{-\frac{1}{2}}\hat{\ell}_1{\Big)}\cdot{\Big(}\Lambda^{-\frac{1}{2}}\hat{\ell}_2{\Big)}
&=-L^2{\Big(}\lambda_{1,K}^{-1}\cos\theta\cos\big(\frac{\pi}{3}+\theta\big)+
\lambda_{2,K}^{-1}\sin\theta\sin\big(\frac{\pi}{3}+\theta\big){\Big)}.
\end{align*}
Thus, we get
\begin{equation*}
\sum_{\substack{i,j=1,2\\i\leq
j}}{\Big(}\Lambda^{-\frac{1}{2}}\hat{\ell}_i{\Big)}\cdot{\Big(}\Lambda^{-\frac{1}{2}}\hat{\ell}_j{\Big)}
=\frac{3}{4}{\Big(}\lambda_{1,K}^{-1}+
\lambda_{2,K}^{-1}{\Big)},
\end{equation*}
which gives to
\begin{equation}\label{Q2part}
\frac{L^4\det(H_K)^{\frac{1}{2}}}{96C_K^2|\hat{K}|}\sum_{\substack{i,j=1,2\\i\leq
j}}{\Big(}\Lambda^{-\frac{1}{2}}\hat{\ell}_i{\Big)}\cdot{\Big(}\Lambda^{-\frac{1}{2}}\hat{\ell}_j{\Big)}
=\frac{L^6\det(H_K)^{\frac{1}{2}}}{128C_K^2|\hat{K}|}\frac{\mbox{tr}(H_K)}{\det(H_K)}
=\frac{L^4\mbox{tr}(H_K)}{32\sqrt{3}C_K^2\det(H_K)^{\frac{1}{2}}}.
\end{equation}
Insert (\ref{Q2part}) into (\ref{Q2total}) and then the formula for $Q_{2,K}$ is proved.

\noindent Similar calculation can produce the corresponding formula for $Q_{3,K}$:
\begin{align*}
Q_{3,K}&=\frac{1}{16|K|^3}{\Big(}{|\bf
\ell}_{2}|^2d_{11}-2{\bf \ell}_{1}^t{\bf \ell}_{2}d_{12}+{|\bf
\ell}_{1}|^2d_{22}{\Big)}^2=\frac{L^4}{16C_K^2|K|^3}{\Big(}\sum_{\substack{i,j=1,2\\i\leq
j}}{\bf \ell}_i^t{\bf \ell}_j{\Big)}^2\\
&=\frac{L^4}{16C_K^4|K|^3}\cdot\frac{9}{16}{\Big(}\lambda_{1,K}^{-1}+
\lambda_{2,K}^{-1}{\Big)}^2=\frac{9L^8}{256C_K^4|K|^3}\frac{\mbox{tr}(H_K)^2}{\det(H_K)^2}
=\frac{3L^4\mbox{tr}(H_K)^2}{16C_K^2|K|\det(H_K)}.
\end{align*}
To analyze the term $\widetilde{Q}_{2,K}$, more patience should be paid. First, similar to that of $Q_{2,K}$,
\begin{equation}\label{widetildeQ2}
\widetilde{Q}_{2,K}=\frac{L^4}{96C_K^2|K|}\sum_{\substack{i,j=1,2\\i\leq j}}k_ik_j.
\end{equation}
Second, using $\hat{\ell}_1=L(\cos\theta,\sin\theta)^t$, we have
\begin{equation*}
k_1=(b_{2},-b_{1}){\bf \ell}_{1}=C_K^{-\frac{1}{2}}(b_{2},-b_{1})R^{-1}\Lambda^{-\frac{1}{2}}\hat{\ell}_1
=C_K^{-\frac{1}{2}}L{\Big(}A_{1}\lambda_{1,K}^{-\frac{1}{2}}\cos\theta
+A_{2}\lambda_{2,K}^{-\frac{1}{2}}\sin\theta{\Big)}.
\end{equation*}
Similarly, we have
\begin{equation*}
k_2=-C_K^{-\frac{1}{2}}L{\Big(}A_{1}\lambda_{1,K}^{-\frac{1}{2}}\cos\big(\frac{\pi}{3}+\theta\big)
+A_{2}\lambda_{2,K}^{-\frac{1}{2}}\sin\big(\frac{\pi}{3}+\theta\big){\Big)}.
\end{equation*}
Thus,
\begin{equation}\label{sumkikj}
\sum_{\substack{i,j=1,2\\i\leq
j}}k_ik_j=\frac{3}{4}C_K^{-1}L^2{\Big(}A_{1}^2\lambda_{1,K}^{-1}
+A_{2}^2\lambda_{2,K}^{-1}{\Big)}.
\end{equation}
Finally, insert (\ref{sumkikj}) into (\ref{widetildeQ2}), it is easily got that
\begin{align*}
\widetilde{Q}_{2,K}&=\frac{L^6}{128C_K^3|K|}{\Big(}\frac{A_1^2}{\lambda_{1,K}}+\frac{A_2^2}{\lambda_{2,K}}{\Big)}
=\frac{L^6\sqrt{\det(H_K)}}{128C_K^2|\hat{K}|}{\Big(}\frac{A_1^2}{\lambda_{1,K}}+\frac{A_2^2}{\lambda_{2,K}}{\Big)}
\\&=\frac{L^4\sqrt{\det(H_K)}}{32\sqrt{3}C_K^2}{\Big(}\frac{A_1^2}{\lambda_{1,K}}+\frac{A_2^2}{\lambda_{2,K}}{\Big)}.
=\frac{L^4}{32\sqrt{3}C_K^2}\cdot\frac{{\bf b}^tH_K{\bf b}}{\sqrt{\det(H_K)}}
\end{align*}
Now, we have obtained all the desired results in Theorem \ref{theorem:Q_iK} and the proof is complete.
\end{proof}

\noindent
To summarize,
\begin{equation*}
E_K=\frac{L^4}{C_K^2}{\Bigg(}\frac{|K|}{15\alpha_K} + \frac{\varepsilon\mbox{tr}(H_K)}{32\sqrt{3}\sqrt{\det(H_K)}} +
\frac{\alpha_K}{32\sqrt{3}}
\frac{{\bf b}^tH_K{\bf b}}{\sqrt{\det(H_K)}}
+\frac{3\varepsilon^2\alpha_K\mbox{tr}(H_K)^2}{16|K|\det(H_K)}{\Bigg)}
=:\frac{L^4}{C_K^2}P(\alpha_K).
\end{equation*}
To minimize the term
\begin{equation*}
P(\alpha_K)=\frac{|K|}{15\alpha_K}+\frac{\alpha_K}{32\sqrt{3}}{\Bigg(}
\frac{{\bf b}^tH_K{\bf b}}{\sqrt{\det(H_K)}}
+\frac{6\sqrt{3}\varepsilon^2\mbox{tr}(H_K)^2}{|K|\det(H_K)}{\Bigg)} + \frac{\varepsilon\mbox{tr}(H_K)}{32\sqrt{3}\sqrt{\det(H_K)}},
\end{equation*}
the following condition should be satisfied
\begin{equation*}
\frac{|K|}{15\alpha_K}=\frac{\alpha_K}{32\sqrt{3}}{\Bigg(}
\frac{{\bf b}^tH_K{\bf b}}{\sqrt{\det(H_K)}}
+\frac{6\sqrt{3}\varepsilon^2\mbox{tr}(H_K)^2}{|K|\det(H_K)}{\Bigg)}.
\end{equation*}
This equation gives the following choice for the optimal stabilized parameters
\begin{equation}\label{2dstab_param}
\alpha_K^*=\sqrt{\frac{32\sqrt{3}}{15}}\cdot|K|\cdot{\Bigg(}\frac{|K|{\bf b}^tH_K{\bf b}}{\sqrt{\det(H_K)}}
+\frac{6\sqrt{3}\varepsilon^2\mbox{tr}(H_K)^2}{\det(H_K)}{\Bigg)}^{-\frac{1}{2}}.
\end{equation}
At this time,
{\begin{align*}
P(\alpha_K^*)&=\sqrt{\frac{1}{120\sqrt{3}}{\Bigg(}
\frac{|K|{\bf b}^tH_K{\bf b}}{\sqrt{\det(H_K)}}
+\frac{6\sqrt{3}\varepsilon^2\mbox{tr}(H_K)^2}{\det(H_K)}{\Bigg)}}
+\frac{\varepsilon\mbox{tr}(H_K)}{32\sqrt{3}\sqrt{\det(H_K)}}\\
&\approx\sqrt{\frac{1}{120\sqrt{3}}{\Bigg(}
\frac{|K|{\bf b}^tH_K{\bf b}}{\sqrt{\det(H_K)}}
+\frac{6\sqrt{3}\varepsilon^2\mbox{tr}(H_K)^2}{\det(H_K)}{\Bigg)}}=\frac{2|K|}{15\alpha_K^*}.
\end{align*}
Here we omit the small term $\frac{\varepsilon\mbox{tr}(H_K)}{32\sqrt{3}\sqrt{\det(H_K)}}$ to simplify the formula of the metric tensor. Of course if we do not omit this term the derivation can be carried out similarly, however, in this case the expression of the metric tensor will seem rather complicated. In fact the numerical efficiency is almost the same no matter this term is omitted or not (we have done numerical experiments to verify this point, however, due to the length reason we do not list the comparison in this paper). To proceed the derivation, at this time
\begin{equation*}
E_K=\frac{L^4}{C_K^2}P(\alpha_K^*)\approx\frac{2L^4|K|}{15\alpha_K^*C_K^2}.
\end{equation*}
To satisfy the equidistribution requirement, we require that
\begin{equation*}
\frac{L^4}{C_K^2}P(\alpha_K^*)\approx\frac{2L^4|K|}{15\alpha_KC_K^2}=\frac{e}{N},
\end{equation*}
where $N$ is the number of elements of $\mathcal{T}_h$. Then $C_K$ could be the form
\begin{equation}\label{relation_C_K_Alpha}
C_K\sim \sqrt{\frac{|K|}{\alpha_K^*}},
\end{equation}
and $M({\bf
x})$ could be the form
\begin{equation}\label{2dstab_monitor}
M({\bf x})=\sqrt[{4}]{\frac{|K|{\bf b}^tH_K{\bf b}}{\sqrt{\det(H_K)}}
+\frac{6\sqrt{3}\varepsilon^2\mbox{tr}(H_K)^2}{\det(H_K)}}H(u).
\end{equation}
To establish the metric tensor $\mathcal{M}({\bf x})$, we just need to combine the formula of the monitor function and the size requirement. Since the procedure is standard, we omit it due to the length reason.

\subsection{Practical use of stabilized parameters}
Since there exists a constant $C$ in error estimate (\ref{2derrorbound}), the exact ratios between the terms $\alpha_K^{-1}\|u-\Pi_hu\|_{L^2(K)}^2$, $\varepsilon\|\nabla (u-\Pi_hu)\|_{L^2(K)}^2$, $\alpha_K\|{\bf b}\cdot\nabla (u-\Pi_hu)\|_{L^2(K)}^2$, and $\alpha_K\varepsilon^2\|\Delta(u-\Pi_hu)\|_{L^2(K)}^2$ can hardly be precisely estimated. So the stabilized parameters $\alpha_K$ (\ref{2dstab_param}) and the monitor function (\ref{2dstab_monitor}) can be regarded just as quasi-optimal. To establish the practical form of the stabilized parameters with the same scale with other stabilized strategies, consider the special case: the true solution is isotropic. Due to (\ref{2dstab_param}), when convection dominates diffusion, that is,
\begin{equation*}
\frac{|K|{\bf b}^tH_K{\bf b}}{\sqrt{\det(H_K)}}\gg\frac{\varepsilon^2\mbox{tr}(H_K)^2}{\det(H_K)},
\end{equation*}
the following estimate holds:
\begin{equation*}
\frac{15}{32\sqrt{3}}\cdot(\alpha_K^{*})^2=\frac{|K|\sqrt{\det(H_K)}}{{\bf b}^tH_K{\bf b}}
=\frac{C_K|K|\sqrt{\det(H_K)}}{{\bf b}^tC_KH_K{\bf b}}=\frac{|\hat{K}|}{{\bf b}^tM_K{\bf b}}
=\frac{|\hat{K}|}{\hat{\bf b}^t\hat{\bf b}}=\frac{|K|}{{\bf b}^t{\bf b}}.
\end{equation*}
And then
\begin{equation*}
\sqrt{\frac{15}{32\sqrt{3}}}\cdot\alpha_K^{*}=\frac{\sqrt{|K|}}{\|{\bf b}\|}=\frac{\sqrt[4]{3}\|{\bf \ell}_{1}\|}{2\|{\bf b}\|}.
\end{equation*}
On the contrary, when diffusion dominates convection,
\begin{equation*}
\frac{|K|{\bf b}^tH_K{\bf b}}{\sqrt{\det(H_K)}}\ll\frac{\varepsilon^2\mbox{tr}(H_K)^2}{\det(H_K)},
\end{equation*}
we have
\begin{equation*}
\sqrt{\frac{15}{32\sqrt{3}}}\cdot\alpha_K^{*}=
\frac{1}{\sqrt{6\sqrt{3}}}\cdot\frac{|K|\sqrt{\det(H_K)}}{\varepsilon\mbox{tr}(H_K)}=
\frac{1}{\sqrt{6\sqrt{3}}}\cdot\frac{\sqrt{3}\|{\bf \ell}_{1}\|^2}{8\varepsilon}.
\end{equation*}
To compare with the other stabilized parameters in the same scale we suggest
{\begin{equation}\label{2dstab_param_prac}
\alpha_K^{*}=|K|\cdot{\Bigg(}\frac{\sqrt{3}\cdot|K|{\bf b}^tH_K{\bf b}}{\sqrt{\det(H_K)}}
+\frac{27\varepsilon^2\mbox{tr}(H_K)^2}{4\det(H_K)}{\Bigg)}^{-\frac{1}{2}},
\end{equation}
and corresponding monitor function $M({\bf
x})$ could be the form
\begin{equation}\label{2dstab_monitor_prac}
M({\bf x})=\sqrt[{4}]{\frac{|K|{\bf b}^tH_K{\bf b}}{\sqrt{\det(H_K)}}
+\frac{9\sqrt{3}\varepsilon^2\mbox{tr}(H_K)^2}{4\det(H_K)}}H(u).
\end{equation}

\begin{figure}[tbhp]
\centering
\subfloat[LEP]{\includegraphics[width=6cm]{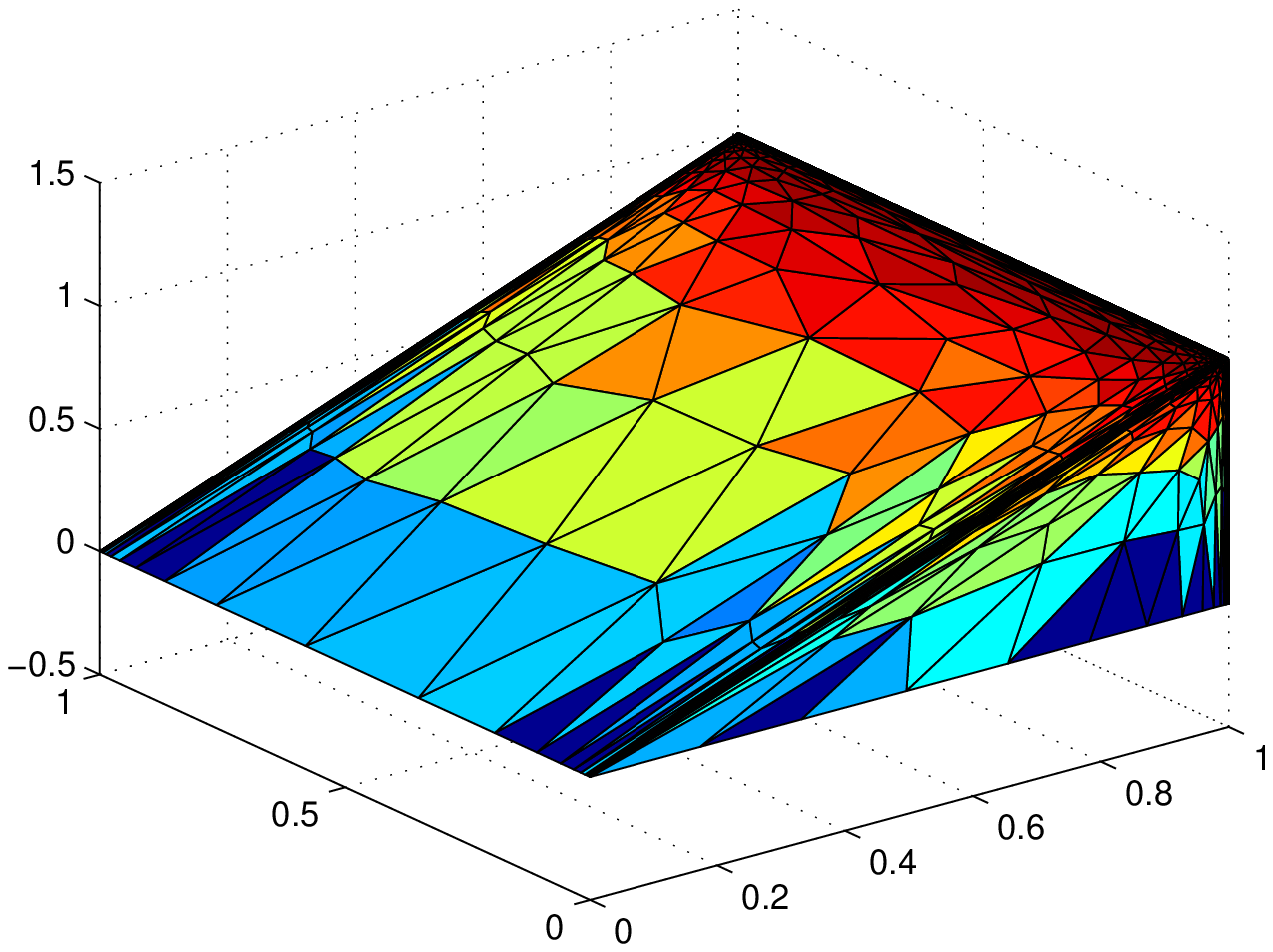}}
\subfloat[PLE]{\includegraphics[width=6cm]{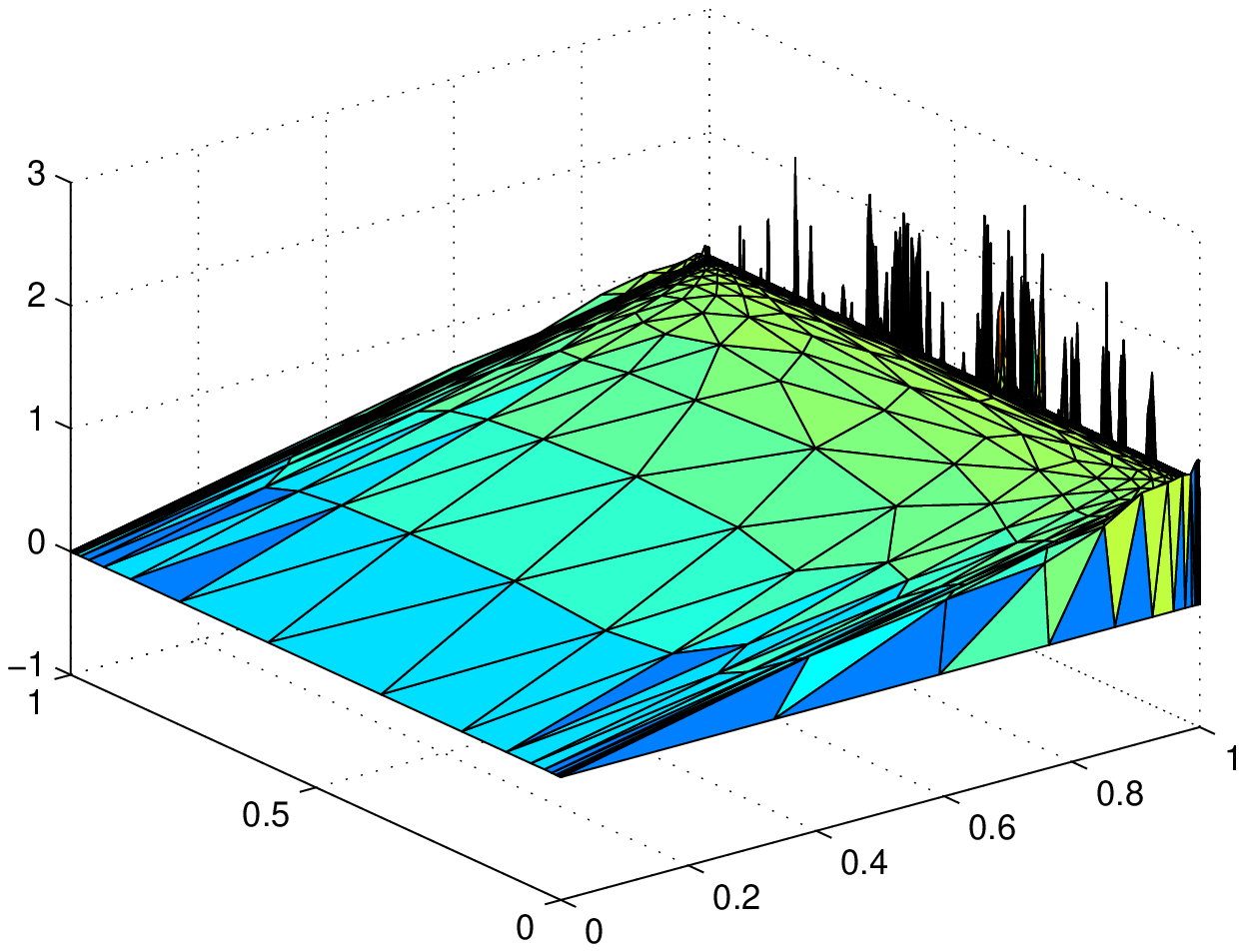}}\\
\subfloat[DDC]{\includegraphics[width=6cm]{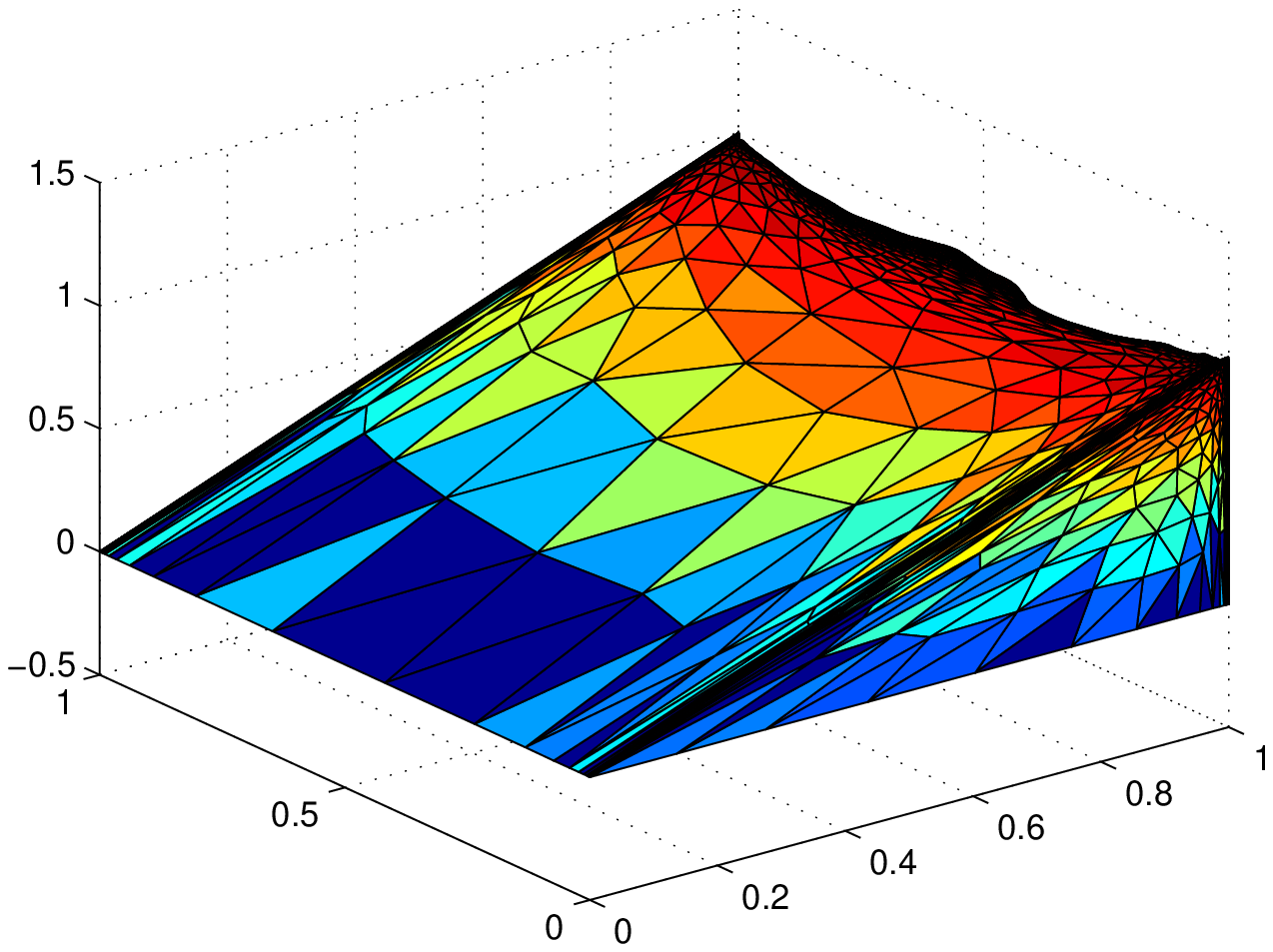}}
\subfloat[NSP]{\includegraphics[width=6cm]{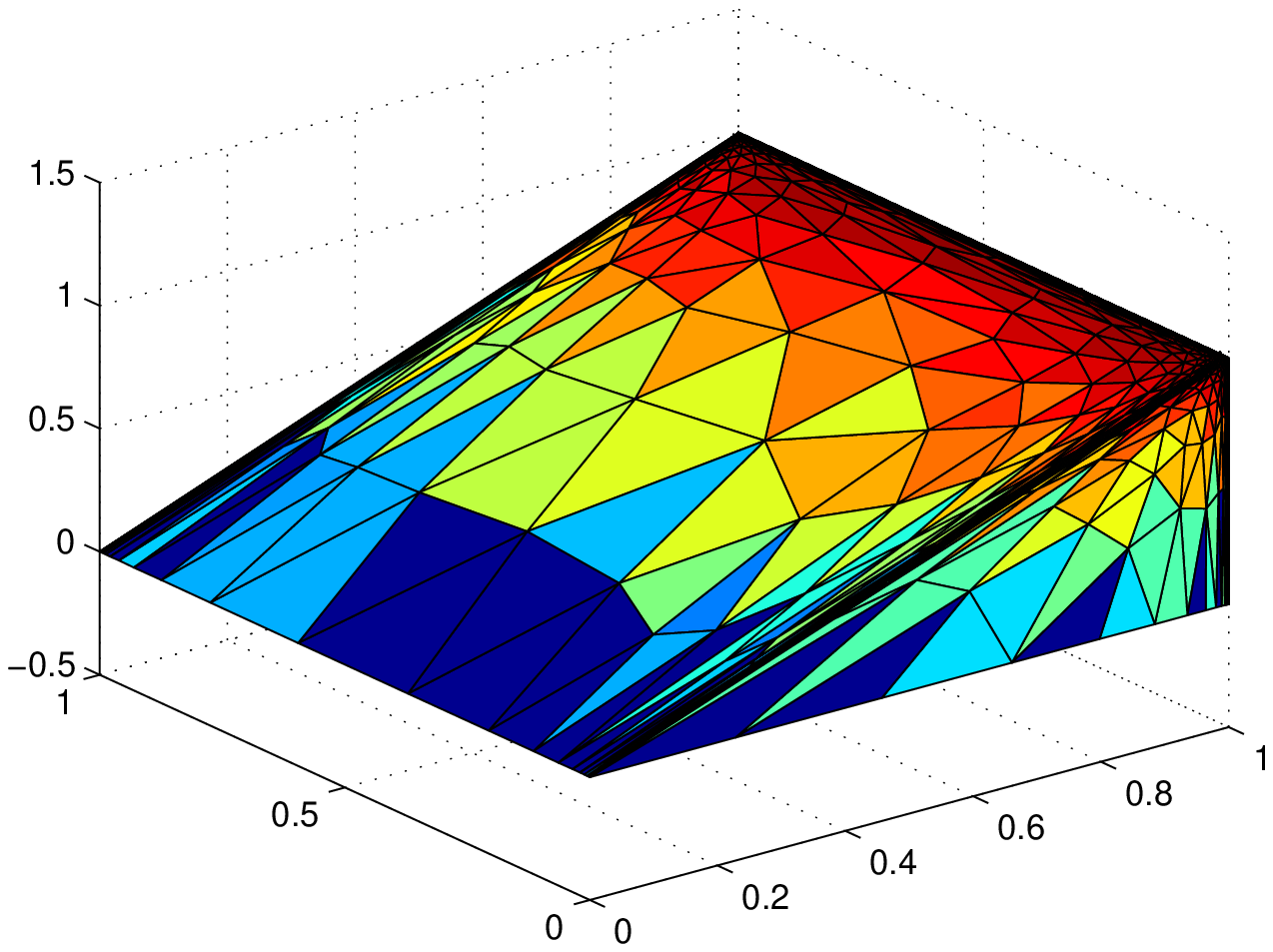}}
\caption{Example \ref{Example_1}: solutions generated by four strategies using monitor function (\ref{2dstab_monitor_prac}).}\label{fig:5-1-1}
\end{figure}
\section{Numerical examples}\label{Section_Numerical_Examples}
In this section, we will demonstrate several numerical examples to see the superiority of our strategy to others. All the presented experiments
are performed using the BAMG project \cite{hecht1998bamg} via software FreeFem++ \cite{hechtfreefem++}. Given a ¡°background¡± mesh, the nodal
values of the solution are obtained
by solving a PDE through the GLS finite element method. Then the metric tensors are obtained by using
some recovery techniques ( In this paper we use the the gradient recovery technique proposed in \cite{zienkiewicz1987simple}). Finally,
a new mesh according to the computed metric tensor is generated by BAMG. The process is repeated several times in the computation until the approximate solution satisfies the prescribed tolerance.

\subsection{Stability vs parameters}
\begin{example}\label{Example_1}
We consider (\ref{adv-diff-problem}) in $\Omega=(0, 1)^2$ with ${\bf b}=(1, 0)^T$, $f=1$ and the homogeneous boundary condition.
The solution for this problem possesses an regular boundary
layer at $x=1$ and parabolic boundary layers at $y=0$ and $y=1$.
This example has been used, e.g., in \cite{john2006computational,mizukami1985petrov}.
\end{example}
Five stabilized strategies using monitor function (\ref{2dstab_monitor_prac}) are compared in terms of stability. Since the stability of stabilized strategies DDC and DEE are similar in this example, here we just list the result for DDC out of the two. It is seen from Figure \ref{fig:5-1-1} that the solution by the stabilized strategy PLE exhibits spurious oscillations. For the solution by the stabilized strategy DDC (and also DEE), although there does not exist obvious oscillation, the regular layer is smeared. While the solutions by the stabilized strategies LEP and NSP have better stability.

\begin{figure}[tbhp]
\centering
\subfloat[DEE]{\includegraphics[width=6cm]{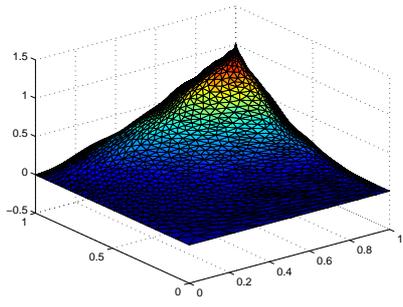}}
\subfloat[LEP]{\includegraphics[width=6cm]{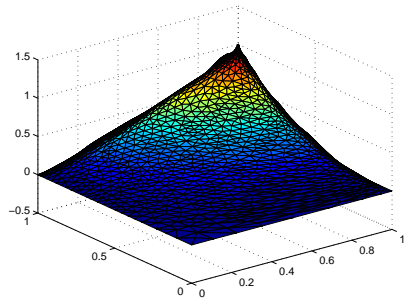}}\\
\subfloat[DDC]{\includegraphics[width=6cm]{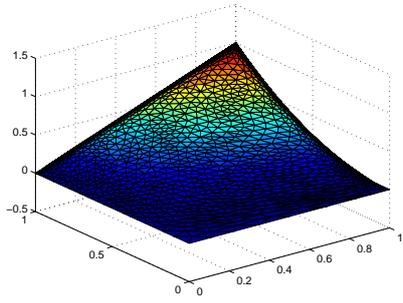}}
\subfloat[NSP]{\includegraphics[width=6cm]{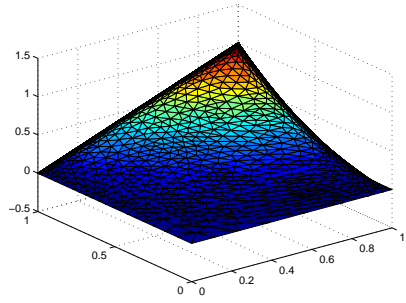}}
\caption{Example \ref{Example_2}: solutions generated by four strategies using monitor function (\ref{2dstab_monitor_prac}).}\label{fig:5-2-1}
\end{figure}

\begin{example}\label{Example_2}
We consider (\ref{adv-diff-problem}) in $\Omega = (0, 1)^2$ with $\varepsilon = 10^{-8}$ and $\mathbf b = (2, 3)^T$. The right-hand side $f$
and the Dirichlet data $g$ are chosen in such a way that
\begin{align*}
u(x,y)=xy^2-y^2\exp\Big(\frac{2(x-1)}{\varepsilon}\Big)-
x\exp\Big(\frac{3(y-1)}{\varepsilon}\Big)+\exp\Big(\frac{2(x-1)+3(y-1)}{\varepsilon}\Big),
\end{align*}
which exhibits regular boundary
layers at $x=1$ and $y=1$. This solution has been used, e.g., in \cite{sun2010numerical} where the convection-diffusion-reaction equations are considered.
\end{example}

Five stabilized strategies using monitor function (\ref{2dstab_monitor_prac}) are compared in terms of stability. Since the stability of stabilized strategies LEP and PLE are similar in this example, here we just list the result for LEP out of two. It is seen from Figure \ref{fig:5-2-1} that the solutions by the stabilized strategies LEP (and also PLE) and DEE smear the regular layer. While the solutions by the stabilized strategies DDC and NSP have better stability.

In this subsection, we compare five stabilized strategies in terms of stability through two examples. NSP have good stability for every example, while other stabilized strategies behave not so good in terms of stability for at least one example.

\subsection{Accuracy vs metric tensors}
Still consider Example \ref{Example_2}, since the exact solution is given, we demonstrate in this subsection that the metric tensor proposed in this paper is more suitable for SFEMs approximating the convection-dominated convection-diffusion equation than those optimal ones for the interpolation error in some norms, e.g. $L^2$ norm (\cite{chen2007optimal}) which is defined in term of monitor function by
\begin{equation}\label{L2_monitor}
M({\bf x})=\frac{1}{\sqrt[6]{\det(H_K)}}H(u).
\end{equation}
Following the analysis of Section \ref{Section_Metric_Tensor}, if the monitor function is set to be (\ref{L2_monitor}), the optimal form of stabilization parameters is also (\ref{2dstab_param_prac}). Five stabilized strategies using monitor function (\ref{L2_monitor}) are compared in terms of stability. The stability for these strategies are similar with that using monitor function (\ref{2dstab_monitor_prac}). In Figure \ref{fig:5-2-2} the discretized errors in $L^2$ norm are compared between five stabilized strategies using monitor functions (\ref{2dstab_monitor_prac}) and (\ref{L2_monitor}).

From Figures \ref{fig:5-2-1} and \ref{fig:5-2-2} we conclude as follows: (1) For every  stabilized strategy, the error in $L^2$ norm using monitor function (\ref{2dstab_monitor_prac}) is smaller than that using (\ref{L2_monitor}). (2) We could divide these stabilized strategies into two types by the convergence of errors in $L^2$ norm. The first type contains the strategies DEE, PLE, and LEP. The second type contains the strategy DDC, and NSP which behaves better than the first one in terms of accuracy no matter which metric tensor is used. (3) For the two stabilized strategies of the second type, NSP behaves better than  DDC when the discretized error in $L^2$ norm is concerned no matter which metric tensor is used.
\begin{figure}[tbhp]
\centering
\subfloat[monitor function (\ref{2dstab_monitor_prac})]{\includegraphics[width=6cm]{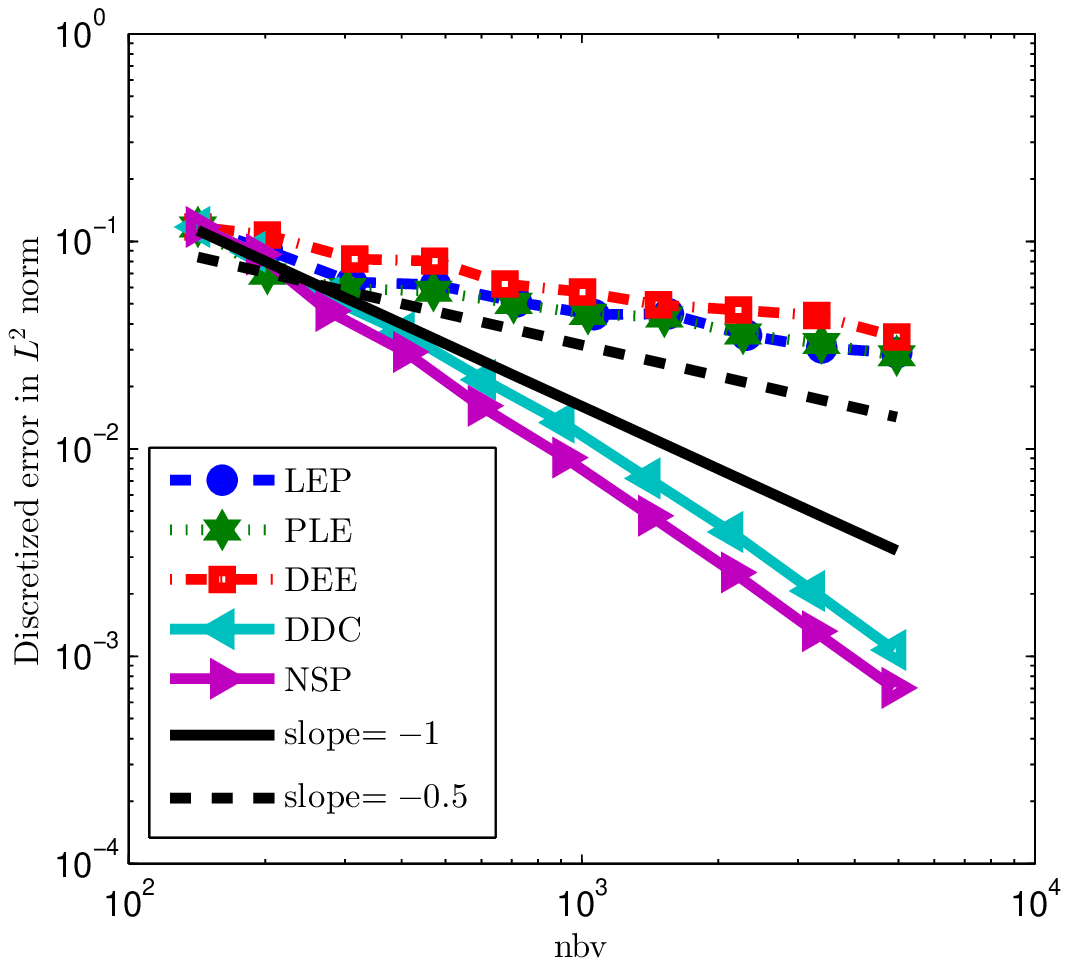}}
\subfloat[monitor function (\ref{L2_monitor})]{\includegraphics[width=6cm]{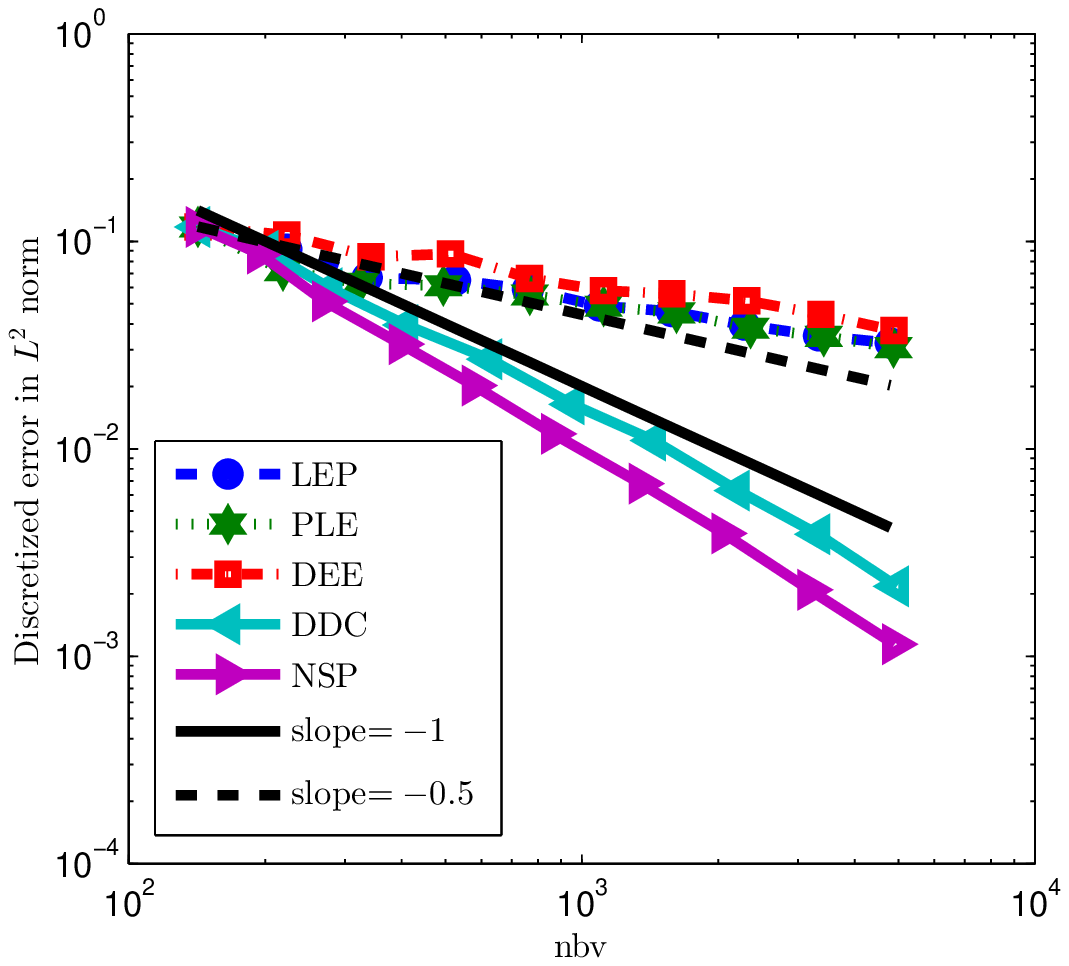}}
\caption{Example \ref{Example_2}: discretized errors in $L^2$ norm using five stabilized strategies.}
\label{fig:5-2-2}
\end{figure}

\subsection{Relationship between our parameter and the best one among others}
In view of the above considerations, besides NSP it seems that strategy DDC is the best choice in term of accuracy. It is then interesting to study the relationship between NSP and DDC, so we give the following analysis.
We set ${\bf b}_0$ as the unit vector in the direction of ${\bf b}$, that is ${\bf b}=|{\bf b}|{\bf b}_0$, then
\begin{equation*}
\frac{|K|{\bf b}^tH_K{\bf b}}{\sqrt{\det(H_K)}}=\frac{|K||{\bf b}|^2(h_K{\bf b}_0)^tM_K(h_K{\bf b}_0)}{h_K^2C_K\sqrt{\det(H_K)}}=\frac{|K|^2|{\bf b}|^2\hat{\bf b}_h^t\hat{\bf b}_h}{h_K^2|\hat{K}|}\in {\Big[}4/\sqrt{3},\sqrt{3}{\Big]}\cdot\frac{|K|^2|{\bf b}|^2}{h_K^2},
\end{equation*}
where $h_K$ is the diameter of $K$ in the direction of the convection ${\bf b}$, $\hat{\bf b}_h$ is the image of $h_K{\bf b}_0$. $\hat{\bf b}_h\in [\sqrt{3}L/2,L]$ holds due to the affine map $\mathcal{F}_K$. Thus, on the triangle where convection dominates diffusion under the standard of NSP,
\begin{equation*}
\alpha_K^*\approx\sqrt[4]{\frac{1}{3}}|K|\cdot{\Bigg(}\frac{|K|{\bf b}^tH_K{\bf b}}{\sqrt{\det(H_K)}}
{\Bigg)}^{-\frac{1}{2}}\in {\Big[}1/2,1/\sqrt{3}{\Big]}\cdot\frac{h_K}{|{\bf b}|},
\end{equation*}
the lower bound is the same as the stabilized strategy DDC, while the ratio between the upper bound and the lower bound is $2/\sqrt{3}$. This result shows that the two stabilized strategies are similar on those triangles where convection dominates diffusion under the standard of NSP. So the main difference between the two strategies comes from those triangles where diffusion dominates convection, indicating that the stabilized strategy on the diffusion-dominated elements plays an important role for SFEMs.
\section{Conclusion}
In this paper, we propose a strategy which generates anisotropic meshes and selects the optimal stabilization parameters for the GLS or SUPG linear finite element methods to solve the two dimensional convection-dominated convection-diffusion equation. This strategy basically solves the two key problems mentioned at the beginning of this paper in a relatively rigorous way, i.e., (1) how to generate optimal anisotropic meshes for minimizing the discretized error of SFEMs for the convection-dominated convection-diffusion equation, and (2) based on (1) how to select optimal stabilization parameters. Numerical examples also indicate that the new strategy proposed in this article is superior than any existed one when both stability and accuracy are concerned.

\end{document}